  \documentclass[11pt]{amsart}

\usepackage[colorlinks,linkcolor={blue}]{hyperref}
\usepackage[none]{hyphenat}

\usepackage{amssymb}
 \usepackage{graphicx} 
 \usepackage{amscd}
 \usepackage{enumerate}
 \usepackage{cancel}

\usepackage[colorinlistoftodos, textsize=footnotesize]{todonotes}
\makeatletter
\providecommand\@dotsep{5}
\makeatother
\swapnumbers

 \def\a{\alpha}

 \def\be{\beta}

 \def\e{\epsilon}
 \def\ve{\varepsilon}

 \def\ga{\gamma}

 \def\la{\lambda}
 
 \def\La{\Lambda}

 \def\re{{\mathbb R}}

 \def\ov{\overline}

 \def\SS{{\mathbb S}}

 \def\tq1{{\tilde{q}_1}}

 \def \ov{\overline}

 \DeclareMathOperator{\diam}{diam}

  \renewcommand{\proofname}{{\bf Proof:}}

  \makeatletter
\def\th@plain{%
  \thm@notefont{}
  \itshape 
}
\def\th@definition{%
  \thm@notefont{}
  \normalfont 
}
\makeatother

 \theoremstyle{plain}

 \newtheorem{Thm}{Theorem}[section]
 
 \newtheorem{Lemma}[Thm]{\bf Lemma}
 
 \newtheorem{Corollary}[Thm]{\bf Corollary}
 \newtheorem{Theorem}[Thm]{\bf Theorem}
 \newtheorem{Proposition}[Thm]{\bf Proposition}

 \newtheorem{Kerekjarto}[Thm]{\bf Ker\'ekj\'art\'o Theorem}

 \newtheorem{claim}[Thm]{\it Claim}
 
 \theoremstyle{definition}

 \theoremstyle{remark}
 \newtheorem{Remark}[Thm]{\bf Remark}

 \newtheoremstyle{Cl}
  {5pt}
  {3pt}
  {\sl}
  {}
  {\it}
  {:}
  {.5em}
  {}

 \theoremstyle{Cl}

 \def\begincproof{
                  \renewcommand{\proofname}{\it Proof:}
                  \begin{proof}
                 }

 \def\endcproof{
                \renewcommand{\qedsymbol}{$\diamondsuit$}
                \end{proof} 
                \renewcommand{\qedsymbol}{\openbox}
                \renewcommand{\proofname}{\bf Proof:}
               }

\makeatletter
\newcommand{\customlabel}[2]{%
   \protected@write \@auxout {}{\string \newlabel {#1}{{#2}{\thepage}{#2}{#1}{}} }%
   \hypertarget{#1}{#2}
}
\makeatother

\makeatletter
\@namedef{subjclassname@2020}{\textup{2020} Mathematics Subject Classification}
\makeatother

\def\ted{\hfill$\triangle$}

 \title[The ideal boundary]{The Ideal Boundary and the Accumulation Lemma} 

\author{Fernando Oliveira}
\address{Fernando Oliveira\newline\indent 
Universidade Federal de Minas Gerais\newline\indent
Av. Ant\^onio Carlos 6627, 
31270-901, Belo Horizonte,
MG, Brasil.}

\author{Gonzalo Contreras}
\address{Gonzalo Contreras\newline\indent 
Centro de Investigaci\'on en Matem\'aticas\newline\indent 
A.P. 402, 36.000, Guanajuato, GTO, Mexico}
\thanks{Gonzalo Contreras is partially supported by CONACYT, Mexico, grant A1-S-10145.}


\begin{document}

\parskip +3pt

\begin{abstract}
Let $S$ be a connected surface possibly with boundary, $\mu$ a finite Borel measure which is positive
on open sets and $f:S\to S$ a homeomorphism preserving $\mu$.
We prove that if $K$ is an $f$-invariant  compact connected subset of $S$ and $L$ is a branch 
of a hyperbolic periodic point of $f$ then $L\cap K\ne\emptyset$ implies $L\subset K$.
This is called the accumulation lemma.
For this we develop a classification of connected surfaces with boundary and 
a characterization of residual domains of compact subsets with finitely many 
connected components in a connected surface with boundary. 
\end{abstract}

\maketitle

\tableofcontents


\section{Introduction.}

A fundamental result in the modern theory of the dynamics of  area preserving
maps is the {\it Accumulation Lemma} which asserts that if $K$ is a compact connected
invariant set and $L$ is a branch of a hyperbolic periodic point $p$ (i.e. a connected
component of $W^s(p)-\{p\}$ or $W^u(p)-\{p\}$) then $L\cap K\ne\emptyset$
implies that $L\subset K$. 

The accumulation lemma is used to prove that for generic area preserving maps
all the branches of a hyperbolic periodic points have the same closure.
This fact in turn allows to obtain certain homoclinic orbits. For example a version
of the accumulation lemma in~\cite{CM2} provided homoclinic orbits which permitted
to prove the existence of Birkhoff sections for all Kupka-Smale Reeb flows.

The accumulation lemma was proved by John Mather in Corollary~8.3 of~\cite{Mat9}
for 
\linebreak
connected surfaces without boundary. A proof for the 2-sphere appears in 
Franks, Le Calvez~\cite{FLC} Lemma~6.1.
Here we present a proof of the accumulation lemma for connected surfaces with
boundary and partially defined maps  in Theorem~\ref{P22}.

We need the version for surfaces with boundary and partially defined area preserving maps 
in order to be able to apply the lemma
to return maps of Birkhoff sections of Reeb flows and holonomy maps on broken book
decompositions of contact 3-manifolds. The lemma allows in~\cite{OC1} to obtain homoclinics 
for Kupka-Smale area preserving maps of surfaces with boundary and in~\cite{CO3}
 to prove that Kupka-Smale geodesic flows of surfaces have homoclinic 
in every hyperbolic closed geodesic.

If $S$ is a surface with boundary and $K\subset S$ is a compact connected set, 
the connected components of $S-K$ are called residual domains of $K$.
The proof of the accumulation lemma uses the topology of residual domains,
so we need to extend the representation of connected surfaces by 
Ker\'ekj\'art\'o~\cite{kerekjarto} and Richards~\cite{Rich}
 to connected surfaces with boundary which we do in Proposition~\ref{P11}.
 
 The ideal boundary points of a surface are also known as topological ends.
 We use the former terminology to distinguish them from the prime end
 compactification of a surface which is also needed in the study of dynamics
 of area preserving maps. The ideal boundary gives a description of the way 
 in which compact subsets divide a complete  surface into unbounded components.
 
The proof of the accumulation lemma for subsets of the sphere $\SS^2$ in \cite[6.1]{FLC}
is quick because in $\SS^2$ a residual domain of a compact connected set is simply connected. 
For general surfaces the proof can be localized but requires the study of  the ideal boundary compactification of residual domains.

We see in Proposition~\ref{P12} that a relatively compact ideal boundary point
of a residual domain 
 has a neighborhood in the ideal boundary compactification which is homeomorphic
 to a disk in $\re^2$. This provides a proof of Proposition~1.1 in Mather~\cite{Mat9}
where it is attributed as a consequence of the main result in Richards~\cite{Rich}.
This is a key element to prove the accumulation lemma.

We use exhaustions from Ahlfors and Sario~\cite{AhSa} to obtain a method for
computing ideal boundary points in Proposition~\ref{P9}. 
We also construct a canonical exhaustion for
\linebreak
 residual domains on manifolds
with boundary in subsection \S~\ref{canonical} that we use to prove the accumulation  
lemma~\ref{P22}.
This technique allows a more elementary exposition.

As a  consequence of the methods developed to deal with the ideal boundary of residual domains
 we obtain a characterization of residual domains on surfaces. 
In Corollary~\ref{C18} we prove that for an open subset $U$ of a connected surface $S$:
its frontier in $S$, $fr_SU$, is a compact set with finitely many components if and only if 
$U$ is a residual domain of a compact set $K$ which has finitely many connected components.
We also see that this result is not valid in other topological spaces $S$ which are not surfaces.

\section{Notation and definitions}

If $B$ is a topological space and $A\subset B$, we use the notation
$int_BA$, $cl_BA$ and $fr_BA$ for the interior, closure and the frontier of $A$ in $B$, respectively.
The boundary of a surface will be denoted by $\partial S$ and we will use $S^\circ= S-\partial S$ 
for its interior as a manifold. By {\it domain} we mean a connected open subset of $S$. 
If $A$ is a closed subset of $S$, a connected component of $S-A$ is called a {\it residual domain} of $A$.

A set $K\subset S$ is {\it relatively compact} if its closure $cl_S K$ is compact.
Observe that a residual domain of a compact subset  $K\subset S$ 
is relatively  compact if and only if it is bounded for some complete
metric on $S$.

By a closed (open) disk we mean a set homeomorphic to a closed (open) ball in the plane.

\section{The ideal boundary.}

Let $S$ be a non-compact connected surface
with a complete metric.
We are going to describe a compactification of $S$
by the addition of its (topological) ends or ideal boundary points.
See \cite{AhSa}, \cite{Mat9}, \cite{Rich}.

An {\it ideal boundary component} of $S$ is a decreasing sequence 
$V_1\supset V_2\supset \cdots$ of non-empty subsets of $S$ such that:
\begin{enumerate}
\item \label{c1}$V_n$ is open in $S$.
\item \label{c2} $V_n$ is connected.
\item \label{c3} $cl_SV_n$ is not compact.
\item \label{c4} $fr_S V_n$ is compact.
\item \label{c5} If $K$ is a compact subset of $S$, then there is $n_0$ 
such that $K\cap V_n=\emptyset$ for $n\ge n_0$. 
\newline
We will refer to this property briefly by saying that the sequence 
$(V_n)$  {\it leaves 
\linebreak
compact subsets of $S$.}
\end{enumerate}

Two ideal boundary components $V_1\supset V_2\supset\cdots$ and
 $V_1'\supset V_2'\supset\cdots$
are equivalent if for every $n$ there is $m$ such that $V_m\subset V_n'$ and vice versa.
An {\it ideal boundary point} is an equivalence class of ideal boundary components.
The set of  ideal boundary points $b(S)$ is called the {\it ideal boundary} of $S$ and the
disjoint union $B(S)=S\sqcup b(S)$ is called the {\it ideal completion} of $S$. 
The ideal boundary gives a description of the way in which compact subsets 
of $S$ divide $S$ into unbounded (i.e non relatively compact) residual domains.

\begin{Remark}\label{R1}\sl
It follows from conditions \eqref{c1} through \eqref{c5} of the definition of an ideal 
boundary component $(V_n)$ that the condition 
\begin{enumerate}
\item[\rm(\customlabel{c5p}{$\ref{c5}'$}\!\!)]
\quad
$\cap_n\, cl_S V_n = \emptyset$.
\end{enumerate}
holds true.
In fact, we could have replaced condition \eqref{c5} by 
\eqref{c5p} 
and have an equivalent definition of ideal boundary components.
It is not true that condition \eqref{c5} is equivalent to $\cap_n V_n=\emptyset$,
as we have seen stated elsewhere.

 Consider polar coordinates $(\theta,r)$ on $S=\re^2-\{(0,0)\}$ and define 
 $$
 V_n=\{(\theta,r)\in S\,:\, r<\tfrac 1n\}\cup\{(\theta,r)\in S\, :\, r<2 ,\; 
 0<\theta<\tfrac 1n\}.
 $$
 Then $(V_n)$ is a decreasing sequence of sets that satisfy 
$\cap_n V_n=\emptyset$, \eqref{c1}, \eqref{c2}, \eqref{c3} and \eqref{c4}
but not \eqref{c5}, because
$$
\cap_n cl_S V_n=\{(\theta,r)\in S\,:\, 0<r\le 2, \;
\theta=0\,\}
$$
and $V_n$ does not leave circles centered at $(0,0)$  and radius $r$ 
with $0<r< 2$.
\end{Remark}

\begin{Remark}\label{R2}
For an ideal boundary component $(V_n)$ it is easy to see that 
for every $n$ there exists $m>n$ such that $fr_S V_m\cap fr_SV_n=\emptyset$.
From this and $V_m\subset V_n$ it follows that 
$cl_S V_m\subset V_n$. Therefore, by taking a subsequence of $(V_n)$, 
we have that $(V_n)$ is equivalent to an ideal 
boundary component $(W_n)$ for which 
$cl_SW_{n+1}\subset W_n$ for every $n$.
\end{Remark}

Let $A$ be an open connected subset of $S$ with $fr_SA$ compact 
and let $A'$ be the set of ideal boundary points of  $S$
whose representing ideal boundary components $(V_n)$ 
satisfy $V_n\subset A$ for sufficiently large $n$.
The collection of all sets $A'$ where $A$ is an open connected
subset of $S$ with $fr_SA$ compact, forms a basis for a topology on $b(S)$.
With this topology $b(S)$ is a compact totally disconnected space
(cf. \cite[ch.~1, \S 36--37]{AhSa}).
On $B(S)$ we use the topology with basis the sets $A^*=A\cup A'$,
where $A\subset S$ is open connected with $fr_S A$ compact.

By a compactification of $S$ we mean a compact topological space $M$
that contains $S$ as an open and dense subset.

A characterization of $B(S)$ as a compactification of $S$ is given by the 
following result:

\begin{Proposition}\quad\label{P3}

The set $B(S)$ is a compactification of $S$ that satisfies the following properties:
\begin{enumerate}
\item\label{p31} $B(S)$ is a locally connected Hausdorff space.
\item\label{p32} $b(S)$ is totally disconnected.
\item\label{p33} $b(S)$ is non separating on $B(S)$, \\
(i.e. if $V\subset B(S)$ is open and connected then $V-b(S)$ is connected).
\end{enumerate}
If $M$ is another compactification of $S$ satisfying these properties, 
then there exists a 
homeomorphism from $M$ onto $B(S)$ which is the identity on $S$.
\end{Proposition}

For a proof see sections~36 and~37 of chapter~1 of~\cite{AhSa}.
Later we will provide a good description of $B(S)$. 
Sometimes we will just write $b=(V_n)$ to say that $b$
is an ideal boundary point of $S$ represented by the ideal boundary
component $(V_n)$. Let $V^*_n=V_n\cup V'_n$.
Then $(V'_n)$ and $(V^*_n)$ are fundamental systems of neighborhoods
for $b$ in $b(S)$ and in $B(S)$  respectively.

Now we  describe a result of Ker\'ekj\'art\'o that gives necessary and sufficient
conditions for non-compact surfaces to be homeomorphic.

Let $b\in b(S)$ be represented by an ideal boundary component $(V_n)$.
We say that $b$ is {\it planar} if $V_n$ is homeomorphic to a subset of the plane 
for all sufficiently large $n$. We say that $b$ is {\it orientable} if $V_n$ is orientable
for all sufficiently large $n$. Let 
\begin{align*}
b'(S)&=\{\text{non planar ideal boundary points of $S$}\},
\\
b''(S)&=\{\text{non orientable ideal boundary points of $S$}\}.
\end{align*}
Clearly $b''(S)\subset b'(S)$.

We say that the surface $S$ has {\it finite genus} if there exists 
a compact bordered surface $K$ contained in $S$
such that $S-K$ is homeomorphic to a subset of the plane.
In this case the genus of $S$ is defined to be the genus of $K$.
Otherwise we say that $S$ has {\it infinite genus}.
The genus of a connected surface with boundary can also be defined 
as the maximum number of disjoint simple closed curves that 
can be embedded in the interior of the surface without disconnecting it.

\begin{Remark}\label{R4}\quad 

\sl The surface with boundary $S$ has finite genus if and only if 
every ideal boundary point of $S$ is planar.
\end{Remark}

Now we are going to define four types of orientability for a non compact surface.

Assume that the surface  $S$ is not orientable. 
We say that $S$ is {\it finitely non orientable} if there exists a compact bordered 
surface $K$ contained in $S$ such that $S-K$ is orientable. Otherwise we say
that $S$ is {\it infinitely non orientable}.
Every compact non orientable surface is the connected sum of a compact orientable
surface and one or two projective planes.
If $S$ is finitely non orientable and $S-K$ is orientable, we say that $S$ is of {\it odd}
or {\it even} non orientability type according to $K$ 
being the connected sum of a compact orientable
surface and one or two projective planes, respectively.

\begin{Kerekjarto}\label{kerekjarto}\quad

Let $S_1$ and $S_2$ be two non compact connected surfaces
which have the same genus and the same orientability type.
Then $S_1$ and $S_2$ are homeomorphic if and only if 
there exists a homeomorphism of $b(S_1)$ onto $b(S_2)$,
such that $b'(S_1)$ and $b''(S_1)$ are mapped onto
$b'(S_2)$ and $b''(S_2)$ respectively.

\end{Kerekjarto}

This is an old result of Ker\'ekj\'art\'o and a complete proof appears in 
Theorem~1 of Richards~\cite{Rich}.

If $K$ is a compact totally disconnected subset of a compact surface $M$,
it is not difficult to show that $b(M-K)$ is homeomorphic to $K$
(for instance, the final part of the proof of Theorem~2 of \cite{Rich} 
contains a detailed proof of this fact).
From this, it follows that $B(M-K)$ is homeomorphic to $M$.
Clearly $b'(M-K)=b''(M-K)=\emptyset$.

Any compact, separable, totally disconnected space $X$ is homeomoprhic
to a subset of the Cantor ternary set. Therefore the same happens 
with the ideal boundary of any non compact surface.

\begin{Proposition}\label{P6}\quad

Let $S$ be a non compact connected surface of finite genus.
Then there exists a compact 
surface $M$ which contains a totally
disconnected compact subset $K$ such that $S$ is 
\linebreak
homeomorphic
to $M-K$ and $b(S)$ is homeomorphic to $K$.
\end{Proposition}

\begin{proof}
We have that $b'(S)=b''(S)=\emptyset$. 
Let $M$ be a compact surface of the same genus
and orientability type as $S$.
The Cantor ternary set can be embedded in $M$ and therefore 
there exists a subset $K$ of $M$ homeomorphic to $b(S)$.
It follows from Ker\'ekj\'art\'o Theorem~\ref{kerekjarto} 
that $S$ is homeomorphic to $M-K$.

\end{proof}

From Proposition~\ref{P6} we inmediately have:

\begin{Proposition}\label{P7}\quad

The ideal completion $B(S)$ of a non compact 
connected surface is a surface if and only if $S$
has finite genus.
\end{Proposition}

Now we would like to describe a method of computing
ideal boundary points.

An {\it exhaustion} of $S$ is an increasing sequence 
$F_1\subset F_2\subset \cdots$ of compact connected
bordered surfaces contained in $S$ such that:
\begin{enumerate}[(a)]
\item\label{aexh} $F_n\subset int_S F_{n+1}$.
\item\label{bexh} $S=\cup_n F_n$.
\item\label{cexh} If $W$ is a connected component of $S-F_n$, then 
$cl_SW$ is a non compact bordered surface whose boundary consists 
of exactly one connected component of $\partial F_n$.
\end{enumerate}

We have the following existence result:

\begin{Proposition}\label{P8}\quad

Every non compact connected surface without boundary admits an exhaustion.
\end{Proposition}

For a proof see the  Theorem in section~29A of chapter~1 in Ahlfors and Sario~\cite{AhSa}.

Let $(F_n)$ be an exhaustion of $S$.
Since $F_n$ is a compact bordered surface, the number of connected components
of $\partial F_n$ is finite. This implies that 
 the number of connected components of $S-F_n$ is finite. 

Denote the connected components of $S-F_n$ by $W^i_n$, where $i$ runs 
over some finite set.
If $m>n$ then $S-F_m\subset S-F_n$ and every $W^j_m$
is contained in some $W^i_n$.
Besides that, every $W^i_n$ contains at least one $W^j_m$,
otherwise $W^i_n\subset F_m$ and $cl_SW^i_n$
would be compact, contradicting the definition of an exhaustion.

Let $W^{i_1}_1\supset W^{i_2}_2\supset \cdots \supset W^{i_n}_n\supset \cdots$ 
be a decreasing sequence of connected components of the sets $S-F_n$.
It is easy to check that $(W^{i_n}_n)$ satisfies the conditions
\eqref{c1}--\eqref{c5} for being
an ideal boundary component of $S$.

\begin{Proposition}\label{P9}\quad

Let $(F_n)$ be an exhaustion of a non compact connected surface $S$.
Then every ideal boundary component of $S$ is equivalent to another of
the form $(W_n)$, where $W_n$ is a connected component of $S-F_n$.
\end{Proposition}

\begin{proof}\quad

Let $(V_k)$ be an ideal boundary component of $S$.
Since $fr_S V_n$ is compact and $F_n$ is an exhaustion
of $S$ we have that 
\begin{equation}\label{mok}
\forall k \quad \exists m_0>k \quad \forall m>m_0\quad
fr_S V_k\subset int_SF_m.
\end{equation}
Let $W_m^*$ be the connected component of $B(S)- F_m$
which contains the ideal boundary point $b=(V_n)\in b(S)$.
Let $W_m := W^*_m- b(S)$. We claim that $W_m$ is connected\footnote{This follows from proposition~\ref{P3}.\eqref{p33}. We include a proof for completeness.}. 

Indeed, suppose that $W_m$ is not connected. Let $A_1$, $A_2$ be disjoint nonempty open subsets of $S$
such that $W_m = A_1 \cup A_2$. The sets  $A^*_i:= A_i \cup A_i'$ are nonempty, open and disjoint in $B(S)$.
Let $c=(U_n)\in W^*_m-W_m\subset b(S)$.   
Since $W^*_m$ is open in $B(S)$ and $U^*_n:= U_n\cup\{c\}$ is a fundamental system
of neighborhoods of $c$ in $B(S)$,
we can choose the sets $U_n$ satisfying $U_n\subset W_n$.
Since $U_n$ is connected, there is one $A_i$ such that $U_n\subset A_i$
for infinitely many $n$.
Then $c\in A_i'\subset A^*_i$.
This shows that $W^*_m =A^*_1\cup A^*_2$ and also that $W^*_m$ is disconnected,
which contradicts its choice.

Observe that $W^*_m$ contains a connected component of $S- F_m$.
Since $W_m$ is connected, we have that 
$$
\text{$W_m$ is a connected component of 
$S- F_m$.}
$$

Since $W^*_m$ contains the ideal boundary point $b=(V_n)$, we 
have that $W_m\cap V_k\ne\emptyset$.
Observe that 
$V_k$ is a connected component of $S-fr_S V_k$.
By \eqref{mok},   $S-F_m\subset S-fr_S V_k$ for all $m>m_0(k)$.
Therefore  
$$
\forall m>m_0(k)\qquad W_m\subset V_k.
$$

Since $fr_S W_m\subset F_m$ is compact, 
the set $W^*_m$ is open in $B(S)$. The condition $b\in W_m^*$
means that there exists $\ell>m$ such that $V_\ell\subset W_m$.

For every $k$ we have obtained $\ell>m>k$ and $V_\ell\subset W_m\subset V_k$,
where $W_m$ is a connected component of $S-F_m$. Then there is a sequence
$$
V_1 \supset W_{m_1}\supset V_{k_2}\supset W_{m_2} \supset V_{k_3}\supset \cdots
$$
where $W_{m_i}$ is a connected component of $S-F_{m_i}$ and in particular $fr_S W_{m_i}$ 
is compact. Then $(W_{m_i})$ is an ideal boundary component equivalent to $b$.

\end{proof}

\begin{Remark}\label{R10}\quad

If for every $n$ the number of connected components of  $S-F_n$ is less than or equal to $k$,
then $S$ has at most $k$ ideal boundary points.
\end{Remark}

Now we would like to describe a result that shows that every non compact connected 
surface can be represented as a sphere with the deletion of a finite or infinite 
closed totally disconnected subset and the addition of a finite or infinite number 
of handles and crosscaps. This can be seen as a generalization of the classical
representation theorem for compact surfaces.

So let $S$ be a compact connected surface.

If $S$ has finite genus then the above description is a consequence of Proposition~\ref{P6}. 
In this case $S$ is homeomorphic to the sphere with a closed totally disconnected subset 
removed and the addition of finitely many handles and crosscaps.

Suppose now that $S$ has infinite genus. This is equivalent to $b'(S)\ne \emptyset$.

Let $S^2 = \re^2\cup\{\infty\}$ be the one point compactification of the plane.
We consider the Cantor ternary set $C$ as the subset of $S^2$ consisting 
of all points $(x,0)$ such that $x$ has a ternary expansion which contains no 1's.

Let $X\supset Y\supset Z$ be subsets of $C$ homeomorphic to $b(S)$, $b'(S)$ and $b''(S)$,
respectively. If we look at $C$ as obtained by the process of removing middle thirds,
then $C=\cap_{n\ge 1} J_n$ where $(J_n)$ is a nested sequence of sets consisting of the
union of pairwise disjoint closed intervals $I^k_n$ of length $\tfrac 1 {3^n}$, $1\le k\le 2^n$.
For each $n$ choose a collection of $2^n$ pairwise disjoint open balls 
$B^k_n$ such that $I^k_n\subset B^k_n$, the centers of $B^k_n$ and $I^k_n$ coincide,
and the balls have the same radius.
We also want that either $B^\ell_{n+1}\cap B^k_n=\emptyset$
or $B^\ell_{n+1}\subset B^k_n$. We still have that 
$C=\cap_{n\ge1} \cup_{1\le k\le 2^n} B^k_n$ and for each $x\in C$ 
there exists a unique sequence $k_n$ such that $\cap_{n\ge 1} B^{k_n}_n=\{x\}$.

Every ball $B^k_n$ contains exactly two balls $B^\ell_{n+1}$ and $B^{\ell+1}_{n+1}$.
If $B^k_n$ contains a point of $Z$ then we choose a closed disk $D$ 
contained in $B^k_n$ and disjoint from the closures of $B^\ell_{n+1}$ and 
$B^{\ell+1}_{n+1}$, and make the connected sum of $S^2$ and 
a projective plane along the boundary of $D$.
If $B^k_n$ contains a point of $Y-Z$, then we choose a closed disk 
$D$ contained in $B^k_n$ and disjoint from the closures of 
$B^\ell_{n+1}$ and $B^{\ell+1}_{n+1}$,
and make the connected sum of $S^2$ and a torus along the boundary of $D$.

Let $M$ be the surface obtained after these connected sums and the deletion of
points of $X$. In $S^2$ points of $Z$ are accumulated by crosscaps and points 
of $Y-Z$ are accumulated by handles.

For each $x\in X$ there exists a unique sequence $k_n$ such that 
$\cap_{n\ge 1} B^{k_n}_n=\{x\}$. This gives a one to one
correspondence between points  $x\in X$ and ideal boundary
components $(B^{k_n}_n)$ of $b(M)$.
It is not difficult to show that this correspondence gives a
homeomorphism from $b(M)$ onto $X$ that takes $b'(M)$
to $Y$ and $b''(M)$ to $Z$.
(The final part of the proof of Theorem~2 of \cite{Rich}
contains a detailed proof of this fact).

The surfaces $S$ and $M$ have the same genus and orientability type.
It follows from Ker\'ekj\'art\'o Theorem~\ref{kerekjarto} that $S$ and $M$ are homeomorphic. 
To summarize, we have the following:

\begin{Proposition}\label{P11}\quad

Let $S$ be a non compact connected surface of infinite genus.
Then $S$ is homeomorphic to a surface obtained from the sphere $S^2$
by removing a totally disconnected closed subset $X$, taking an infinite 
collection of pairwise disjoint closed balls $(B_n)$ and making the connected 
sum of $S^2$ and a torus or a projective plane along the boundaries 
of $B_n$. Given a neighborhood $W$ of $X$ in $S^2$ all but finitely many balls
$B_n$ are contained in $W$.
\end{Proposition}

See Theorem~3 of \cite{Rich} and the discussion preceding it.
Although $B(S)$ is not always a surface, the last proposition
gives a very good description of what it looks like.
The completion $B(S)$ is locally homeomorphic to $\re^2$ at every point not in 
$b'(S)$ or $b''(S)$, whose points are accumulated by handles or crosscaps,
respectively.

\subsection{More general surfaces.}\quad

The concept of ideal boundary can be generalized to non connected 
surfaces possibly with boundary.

Let $S$ be a surface. The definitions of ideal boundary components, ideal boundary points,
ideal boundary $b(S)$, ideal completion $B(S)$ and their topologies are the same.
If $S_1$, $S_2$, ... are the connected components of $S$,
then $b(S)=\sqcup_i b(S_i)$ and $B(S)=\sqcup_iB(S_i)$.
The completion $B(S)$ is compact if and only if $S$ has finitely many connected components.
Obviously if $S$ is compact then $b(S)=\emptyset$.

An important case is that of a connected surface $S$ with $\partial S$ compact. 
In this case $\partial S$ is the union of a finite number of curves homeomorphic to 
circles. All results previously developed for non compact connected surfaces without boundary, 
with the obvious 
\linebreak
adaptations, apply to these surfaces. The ideal completion of $S$
is a surface if and only if $S$ has finite genus.

Let $S_1$ and $S_2$ be two non compact connected surfaces with compact boundary
which have the same genus and the same orientability type. 
Then $S_1$ and $S_2$ are homeomorphic if an only if they 
have the same number of boundary components and there 
exists a homeomorphism of $b(S_1)$ onto $b(S_2)$, such that 
$b'(S_1)$ and $b''(S_2)$ are mapped onto $b'(S_2)$ and $b''(S_2)$, respectively.

The definition of exhaustions $(F_n)$ for non compact connected surfaces $S$ 
with compact boundary could be the same.
Since $(int_SF_n)$ is an open cover of $S$, 
any compact subset of $S$ would be contained in some $F_n$.
So for surfaces with compact boundary we will further require that
all the sets $F_n$ in an exhaustion of $S$ contain $\partial S$
in their interiors.
Let $S^*$ be the surface obtained by gluing a closed disk $D_i$ 
to each boundary component $\xi_i$ of $S$.
We are going to think of $S$ as a subset of $S^*$
and write 
\begin{equation}\label{s*}
S^*=S\cup D, \quad \text{ where  }
D=D_1\cup \cdots\cup D_k.
\end{equation}
Obviously $b(S)=b(S^*)$.
Exhaustions of $S$ would correspond to exhaustions of $S^*$
whose compact sets contain $D$ in their interiors.
From this, it is clear that there exist  exhaustions $(F_n)$ for $S$ and
that ideal boundary points of $S$ can be computed by using ideal boundary 
components made of connected components of the sets $S-F_n$.

\subsection{The impression of an ideal boundary point of a surface
contained in another surface.}
\quad

Now let $S$ be a connected boundaryless surface,
let $U\subset S$ be a surface possibly with compact 
boundary and $b$ and ideal boundary point of $U$.
Let $(V_n)$ be an ideal boundary component of $U$ representing  $b$.
We define the {\it impression of $b$} (relative to $S$) as
\begin{equation}\label{Zb}
Z(b):=\textstyle\bigcap _n cl_{B(S)}V_n.
\end{equation}

The impression $Z(b)$ is just the set of limit points in
 $B(S)$ of sequences in $U$ that converge to $b$ in $B(U)$.
 The definition does not depend on the choice of $(V_n)$ and $Z(b)$
 is a nonempty, connected, compact subset of $fr_{B(S)}U$.
 
 We say that $b$ {\it is relatively compact in $S$} 
 if some $V_n$ is relatively compact in $S$, i.e.
 if some $V_n$ has compact closure in $S$.
 It is easy to check that $b$ is relatively compact in $S$ if and only if 
 $Z(b)\subset S$.
 In this case $Z(b)=\cap_n cl_SV_n$.
 We call $b$ a {\it regular ideal boundary point of $U$} if $b$ is relatively compact
 in $S$ and $Z(b)$ contains more than one point. 
 
 Next we show that if $b$ is relatively compact in $S$ then 
 it has a neighborhood in $B(U)$ homeomorphic to $\re^2$.
 
 \begin{Proposition}\label{P12}
 \quad
 
 Let $S$ be a non compact connected surface,
 let $U\subset S$ be a connected surface possibly with compact boundary.
 If $b$ is an ideal boundary point of $U$ which is relatively compact in $S$
  then $b$ has a neighborhood in $B(U)$ homeomorphic to $\re^2$.
 \end{Proposition}

 \pagebreak
 \begin{proof}
 \quad

There exists an ideal boundary component $(V_n)$ of $U$ that 
represents $b$ such that $cl_SV_n$ is compact for some value of $n$.
Consider an exhaustion $(F_n)$ of $S$. 
Then $cl_SV_n\subset int_S F_k$ for some $k$.
Since $F_k$ is a compact surface we have that $V_n$ is a surface
of finite genus.

We claim that $(V_i)_{i\ge n}$ is an ideal boundary component of $V_n$.
It follows from the claim and Remark~\ref{R4} that $V_i$
is homeomorphic to an open subset of $\re^2$ for every $i$ large enough.

By Remark~\ref{R2} 
we may assume that 
$cl_U V_{i+1}\subset V_i$ for every $i$, 
implying that 
\linebreak
$cl_{V_n} V_i = cl_U V_i$ and $fr_{V_n}V_i=fr_U V_i$ 
for $i>n$. 
The proof of the claim is a simple verification that $(V_i)_{i\ge n}$
satisfies conditions~\eqref{c1} through~\eqref{c4} and~\eqref{c5p} 
of the definition 
of ideal boundary components of $V_n$.

Then
 Proposition~\ref{P7} finishes the proof.

 \end{proof}

\subsection{Residual domains of compact sets whose frontier contains finitely
many components.}\quad

Now we are going to present some properties of residual domains of compact subsets of $S$ 
that have finitely many components. Endow $S$ with a complete riemannian metric,
so that the non relatively compact ideal boundary points of a residual domain
are at infinity and its non relatively compact boundary components are unbounded.

\begin{Proposition}\label{P13}\quad

Let $K$ be a compact subset of a complete  connected surface $S$.
Then the union of $K$ and its bounded residual domains is compact.
\end{Proposition}

\begin{proof}\quad

Let $F$ be a compact bordered surface which is a neighborhood of $K$ in $S$.
Let $V$ be the union of the residual domains of $K$ contained in $F$.
Then  $K\cup V$ is closed in $S$,
 and since $K\cup V\subset F$, we have that $K\cup V$ is compact.

Every residual domain of $K$ not contained in $F$ contains at least one connected
\linebreak
component of $\partial F$. 
Therefore at most finitely many residual domains of $K$ are not contained in $F$.
 Since $S$ is complete, any bounded set in $S$ is relatively compact.
Let $W$ be the union of the bounded residual domains of $K$ not contained in $F$.
Then $K\cup V\cup W$ is closed in $S$ and since $W$ is a finite union of relatively compact sets,
we have that $K\cup V\cup W$ is compact.

\end{proof}

\begin{Proposition}\label{P14}\quad

Let $K$ be a compact subset of a connected surface $S$. Then $b(S)$ is  homeomorphic
to a compact subset of $b(S-K)$.
\end{Proposition}

\begin{proof}\quad

It follows from Remark~\ref{R2} that any ideal boundary point $b\in b(S)$ can be represented 
by an ideal boundary component $(V_n)$ such that $cl_S V_n\subset S-K$.
We have that $cl_{S-K}V_n=(S-K)\cap cl_S V_n$ and since $S-K$ is open in $S$
we have that $fr_{S-K} V_n=(S-K)\cap fr_S V_n$.
From this it follows that $cl_{S-K} V_n = cl_S V_n$,
$fr_{S-K} V_n = fr_S V_n$ and the sequence of sets $(V_n)$ 
also defines an ideal boundary point $b'\in b(S-K)$. This correspondence
provides an one-to-one mapping $\phi:b(S)\to b(S-K)$, $\phi(b)=b'$.

Basic sets for the topology of $b(S-K)$ are defined by means of connected
subsets $A$ of $S-K$, such that $A$ is open in $S-K$ and $fr_{S-K}A$ is compact.
Since $fr_SA=fr_{S-K}A\cup (fr_S A\cap K)$, we also have that $A$ is open in $S$, 
connected and $fr_SA$ is compact. Therefore the sets $A$ define basic open sets for the
topologies of $b(S)$ and $b(S-K)$, which we denote by $A'_1$ and $A'_2$, respectively.
From the definition of ideal boundary component it follows that $\phi^{-1}(A'_2)=A'_1$, hence
$\phi$ is continuous.
Since $b(S)$ is compact, it follows that $\phi$ is a homeomorphism from $b(S)$ onto
a compact subset of $b(S-K)$.

\end{proof}

We will just think of $b(S)$ as a subset of $b(S-K)$.
Later in Corollary~\ref{C325} we will see that if $K$ has finitely many connected components then 
$b(S-K)-b(S)$ is a discrete topological space and that $b(U)-b(S)$
is a finite set for every residual domain $U$ of $K$.

First we consider the case when $S$ is a compact surface possibly 
with boundary.

\begin{Proposition}\label{P15}
\quad

Let $S$ be a compact connected surface possibly with boundary and $K$ 
a compact subset of the interior $S^\circ$. Assume that $K$ has $m$ connected components.
If $U$ is a residual domain of $K$, then $U$ has at most $m(g+1)$
 ideal boundary points, where $g$ is the genus of $S$.
\end{Proposition}

\begin{proof}\quad

By taking the union of $K$ with all its residual domains different from $U$,
we may assume that $K\cup U=S$.

Since $K\subset S^\circ$, $U$ is a submanifold of $S$. 
Let $(F_n)$ be a exhaustion of $U$.

We first prove  that if $W$ is a connected component of $S-F_n$,
then $W$ contains at least one of the connected components of $K$.
Since $K\subset S-F_n$, it is enough to show that $W\cap K\ne \emptyset$.
Let us assume by contradiction that $W\cap K=\emptyset$.
It follows that $W\subset U$ and that $W$ is contained in a connected component $V$
of $U-F_n$. But every component of $U-F_n$ is contained in a component of $S-F_n$.
From this we conclude that $W=V$. 
We have that $\partial U=\partial S\cap  U\subset F^\circ_n$ and then 
$cl_SW\cap \partial S=\emptyset$.
Also
$cl_SW=W\cup C_1\cup\cdots \cup C_\ell$, where each $C_i$ is a connected
component of $\partial F_n$.
Therefore $cl_SW\subset U$ and $cl_UV=U\cap cl_SV=U\cap cl_SW=cl_SW$ is compact.
But since $(F_n)$ is an exhaustion of $U$, we have that a connected component $V$
of $U-F_n$ is not relatively compact in $U$, a contradiction.

Let $E_1,\ldots,E_k$ be  the closures in $S$ of the  connected components of $S-F_n$.
Since every component of $S-F_n$ contains a component of $K$,
we have that 
$$
k\le m.
$$

We claim that the boundaries  
\begin{equation}\label{bedis}
\text{$\partial E_i$ are disjoint.}
\end{equation}
Indeed, we have that $\partial E_j\subset \partial (S-F_n)\subset \partial S\cup \partial F_n$.
A connected component of $\partial E_i\cap \partial E_j$ would be a component $D$
of $\partial F_n$ which is the boundary of both $E_i$ and $E_j$ (whose interior is disjoint).
Then $E_i\cup E_j$ contains a tubular neighborhood of $D\subset F_n$.
This contradicts the fact that $F_n$ is a surface.

Let $\nu_i$ be the number of components of $\partial E_i-\partial S$.
We claim that 
$$
\nu_i\le g+1.
$$
In fact, if we had $\nu_i\ge g+2$,
then by removing $g+1$ components $C_1,\ldots,C_{g+1}$ of $\partial E_i-\partial S$
from $S$, we would obtain a disconnected set $S-(C_1\cup\cdots\cup C_{g+1})$.
On the other hand, at least one component $C$ of $\partial E_i-\partial S$ remains in
$S-(C_1\cup\cdots\cup C_{g+1})$. 

By \eqref{bedis} the connected components $E_j$, $j\ne i$, attach to $F_n$
through boundary components which are disjoint from $C_1,\ldots, C_{\nu_i}$.
Therefore  the sets $E^\circ_i\cup(\partial S\cap \partial E_i)$  and
$$
F_n\cup(\cup_{j\ne i}E_j) -(C_1\cup\cdots\cup C_{g+1})
$$
 are connected
and their union is $S-(C_1\cup\cdots\cup C_{g+1})$.
But inside $S-(C_1\cup\cdots\cup C_{g+1})$ these sets are glued through $C$.
Therefore $S-(C_1\cup\cdots\cup C_{g+1})$ is connected, a contradiction.

Let $\eta_n$ be the number of connected components of $\partial F_n-\partial S$.
Since 
$$
\partial F_n-\partial S\subset \partial(S-F_n)-\partial S\subset \cup_{i=1}^k (\partial E_i-\partial S),
$$
we have that 
$$
\eta_n\le \nu_1+\cdots+\nu_k\le k(g+1)\le m(g+1).
$$
Observe that every component of $U-F_n$ contains a component of $\partial F_n-\partial S$.
It follows form Remark~\ref{R10} that $U$ has at most $m(g+1)$
ideal boundary points.

\end{proof}

\begin{Remark}\label{r316}\quad

If $K\cap\partial S\ne\emptyset$ we can apply Proposition~\ref{P15} to
the augmented boundaryless surface $S^*$ from~\eqref{s*} and to 
$K^*=K\cup(S^*-S^\circ)$.
The interior of the connected components of $S-K$ and $S^*-K^*$ are the same.
If $K$ has $m$ connected components, $\partial S$ has $n$ component,
$g$ is the genus of $S$ and $U^*$ is a
connected component of $S^*-K^*$ (or of $S^\circ-K$) then 
$$\# b(U^*)\le (m+n) (g+1)$$
because $K^*$ has at most $m+n$ connected components.
\end{Remark}

\subsection{The canonical exhaustion of an unbounded residual domain.}
\label{canonical}
\quad

If $S$ is a surface with compact boundary and $K\subset S$ is a compact subset
with \linebreak
$K\cap \partial S\ne \emptyset$ then a connected $U$ component of 
$S-K$ may not be a surface with compact boundary.
In this case we can use the augmented boundaryless surface $S^*$
from~\eqref{s*} and $K^*=K\cup(S^*-S^\circ)$ as in Remark~\ref{r316}.

Let $S$ be a non compact connected boundaryless surface.
Endow $S$ with a complete metric.
Let $K$ be a compact subset of $S$ and
let $U$ be a residual domain of $K$,
i.e. a connected component of $S-K$.

We are going to construct an adapted exhaustion of $U$
in order to express its ideal boundary points in terms of
ideal boundary points of $S$ and ideal boundary points of $U\cap N$,
where $N$ is a compact surface which is a neighborhood of $K$ in $S$.
The ideal boundary points of $U$ that are relatively compact in $S$ will correspond
to the ideal boundary points of $U\cap N$ and the ideal boundary points of $U$ 
that are not relatively compact in $S$ will correspond to the ideal boundary points of $S$.

The construction will be applicable to both cases, $U$ unbounded or not.

As a consequence, we will see that if $K$ has finitely many components 
then $U$ has only finitely many ideal boundary points which are relatively
compact in $S$ and that the others are ideal boundary points of $S$.
As a corollary we have that $fr_SU$ has finitely many connected components,
a result that is not valid in all topological spaces.

Let $(F_n)_{n\ge 0}$ be an exhaustion of $S$. 
By Proposition~\ref{P13}, we may assume that
$K$ and its bounded residual domains are contained in $F_0^\circ$.
Therefore a residual domain of $K$ is unbounded if and only if 
it is not contained in $F_0$.

\begin{claim}\label{24c0}\quad
 
Any unbounded component $U$ of $S-K$ contains at least one component of $\partial F_0$.
\end{claim}
\noindent{\it Proof:}

We have that  $U$ intersects both $F_0$ (since $fr_SU\subset K\subset F^\circ_0$) 
and $S-F_0$, implying that $U$ intersects $\partial F_0$.
Therefore some component $C$ of $\partial F_0$ intersects $U$,
and since $\partial F_0\subset S-K$ we have that $C\subset U$.

\qed

From this we conclude that $K$ has finitely many unbounded residual domains.
Let 
\begin{equation}\label{uu}
U_+=U-F^\circ_0
\qquad\text{ and }\qquad
U_-=U\cap F_0.
\end{equation}
Then $U_+$ and $U_-$ are surfaces with compact boundary
and their boundary as surfaces is 
\begin{equation}\label{bu}
\partial U_+=\partial U_-=U\cap \partial F_0=U_+\cap U_-.
\end{equation}
Observe that 
\begin{equation}\label{buiu}
U\cap\partial F_0=\partial U_+\subset U_+.
\end{equation}
Denote the connected components
of $\partial U_+=\partial U_-=U\cap \partial F_0$
by $\xi_1,\ldots,\xi_s$.

\begin{claim}\label{24c1}\quad

The components of $U_+=U-F^\circ_0$ are the 
components of $S-F^\circ_0$ contained in $U$.
\end{claim} 
\noindent{\it Proof:} 

Let $C$ be a component of $U-F^\circ_0$.
We have that $U-F^\circ_0\subset S-F^\circ_0$
and therefore $C$ is contained in a component $D$
of $S-F^\circ_0$.
We will show that $C=D$.
We first prove that $D$ is contained in $U$.
Indeed, since $S-F^\circ_0\subset S-K$ we have that
$D$ is contained in a component of $S-K$ such as $U$.
But $C\subset D$ and $C\subset U$ implies that 
$D\cap U\ne \emptyset$ and therefore $D\subset U$.
Now we have that $D\subset U-F^\circ_0$ and therefore
$D$ must be contained in a component of $U-F^\circ_0$.
Since $C\cap D\ne\emptyset$, we have that $C=D$.
This proves the claim.

\qed

We have that 
$F_0$ is part of an exhaustion of $S$.
Therefore every component of $S-F^\circ_0$
is a surface whose boundary consists exactly of one
component of $\partial F_0$.
From Claims~\ref{24c0} and~\ref{24c1} we conclude that $U_+$ has $s$ connected components,
each one a surface with compact boundary consisting of exactly
one of the curves $\xi_k$.

\begin{claim}\label{24c2} The set $U_-=U\cap F_0$ is connected.
\end{claim}
\noindent{\it Proof:}

Assume by contradiction that $U_-=A_1\cup A_2$ with $A_1\cap A_2=\emptyset$
and $A_i$ closed in $U_-$ and non empty.
Then 
$A_i$ is also closed in $U$.
Also 
$U_-$ contains all the components $\xi_1,\ldots,\xi_s$ of $\partial U_-=U\cap \partial F_0$.
Every $\xi_k$, being connected,  is contained in either $A_1$ or $A_2$.
Let $B_i$ be the union of the components of $U_+$ 
whose boundary curves $\xi_k$ are contained in $A_i$.
Since every component of $U_+$ contains exactly one $\xi_k$,
we have that $B_1\cap B_2=\emptyset$. Also each
$B_i$ is closed in $U$.
Therefore $A_i\cup B_i$, $i=1,2$ are disjoint 
subsets of $U$ that are closed in $U$ whose union is $U$.
That would imply that $U$ is disconnected,
a contradiction.

\qed

\begin{claim}\label{cu+fn}
$U_+\cap F_n= U\cap (F_n-F_0^\circ)$ is compact.
\end{claim}

\noindent{\it Proof:}

From~\eqref{uu}, $U_+=U-F^\circ_0$.
This implies that $U_+\cap F_n= U\cap (F_n-F^\circ_0)$.
Since $F_n-F^\circ_0$ is compact, it is enough to prove that
$U\cap (F_n-F^\circ_0)$ is closed in $S$. 
Since $K\subset F^\circ_0$ and $fr_S U \subset K$
 we have that 
$$
fr_S(U)\cap (F_n-F^\circ_0)=\emptyset.
$$
Therefore 
$$
fr_S(U\cap(F_n-F^\circ_0))\subset
 [fr_S U \cap (F_n-F^\circ_0)] \cup [ U\cap fr_S(F_n-F^\circ_0)]
\subset U\cap (F_n-F^\circ_0).
$$
Hence $U\cap (F_n-F^\circ_0)$ is closed in $S$.

\qed

We construct the exhaustion of $U$ as follows.
Let $(E_n)_{n\ge 1}$ be an exhaustion of $U_-$.
We may assume that $E_1$ contains
 all the components of $\partial U_-=U\cap \partial F_0$. Then
 $E_n$ is a compact surface whose boundary
 contains all the components of $\partial U_-=U\cap \partial F_0$.

For $n\ge 1$, let 
\begin{equation}\label{Gn2}
G_n=E_n\cup(U_+\cap F_n)
\subseteq E_n\cup(F_n-F^\circ_0).
\end{equation}
Using Claim~\ref{cu+fn} we have that $G_n$ is compact. Also $\cup_{n\ge 1}G_n=U$.

\begin{claim} \label{c320}
Every component of 
$U_+\cap F_n\subseteq F_n-F^\circ_0$ contains a component of $\partial F_0$.
\end{claim}
\noindent{\it Proof:}
Let $C$ be a component of $F_n-F^\circ_0$ and let $x\in C$.
There exists a path $\a$ in $F_n$ from $x$ to a point in $\partial F_0$.
Let $y$ be the first point of $\a$ to intersect $F_0$ starting from $x$.
The restriction $\be$ of $\a$ from $x$ to $y$ is a path in $F_n-F^\circ_0$ 
connecting $x$ to a component $\nu$ of $\partial F_0$. 
Therefore $\{x\}\cup\be\cup\nu\subset C$, which proves that $C$ contains a component of $\partial F_0$.

We showed in Claim~\ref{24c1} that a component of $U_+$ is a component of $S-F^\circ_0$.
Therefore a component of $U_+\cap F_n$ is a component
of $(S-F^\circ_0)\cap F_n=F_n-F^\circ_0$.
Using the previous paragraph we conclude that every component of $U_+\cap F_n$ 
contains one component of $U\cap \partial F_0$. 

\qed

\pagebreak
\begin{claim} $G_n$ is connected.
\end{claim}
\noindent{\it Proof:}
We have that $E_n$ is connected and contains all the
components $\xi_1,\ldots,\xi_s$ of $U\cap \partial F_0$.
The inclusion~\eqref{buiu} implies that
$$
U\cap \partial F_0\subset U_+\cap F_n.
$$
Let $C_k$ be the component of $U_+\cap F_n$ that contains $\xi_k$,
then $E_n\cup C_k$ is connected.
By Claim~\ref{c320} these $C_k$'s are all the components of $U_+\cap F_n$.
This implies that 
\begin{equation*}
G_n=E_n\cup (U_+\cap F_n)
=(E_n\cup C_1)\cup\cdots\cup(E_n\cup C_s)
\end{equation*}
is connected.
\qed

\begin{claim}$G_n\subset int_UG_{n+1}$.
\end{claim}
\noindent{\it Proof:}
First we prove  that using~\eqref{uu}, \eqref{Gn2},
\begin{equation}\label{ug}
U^\circ_-\cap G_n \subset E_n,
\qquad
U^\circ_+\cap G_n \subset F_n.
\end{equation}
The first inclusion is because $U^\circ_-\cap U_+=\emptyset$ and 
the second because $U^\circ_+\cap E_n\subset U^\circ_+\cap U_-=\emptyset$.

Let $x\in G_n$.
We will prove that there exists a neighborhood
 $V$ of $x$ in $S$ such that $V\subset G_{n+1}$.
 We have three possibilities for $x$: either
 $x\in U^\circ_-$, $x\in U^\circ_+$ or $x\in\partial U_+=\partial U_-$.
 \begin{enumerate}[(a)]
 \item
 If $x\in U^\circ_-\cap G_n$ 
 then there exists a neighborhood $V$ of $x$
 in $S$ such that $V\subset U^\circ_-$.
 By~\eqref{ug} in this case $x\in E_n\subset int_{U_-}E_{n+1}$,
 and therefore there exists a neighborhood
 $W$ of $x$  in $S$ such that 
 $W\cap U_-\subset E_{n+1}$.
 It follows that $V\cap W\subset E_{n+1}\subset G_{n+1}$.
 \item
 If $x\in U^\circ_+\cap G_n$, by \eqref{ug} we have that $x\in F_n\subset int_S F_{n+1}$.
 Then there is a neighborhood $V$ of $x$ in $S$ such that 
 $V\subset U_+\cap F_{n+1}\subset G_{n+1}$.
 \item Suppose that 
 $x\in G_n$ and $x\in \partial U_+=\partial U_-=U\cap \partial F_0$.
 From the choice of $E_n\supset U\cap  \partial F_0$ we have that
 $x\in E_n\subset int_{U_-}E_{n+1}$.
 Then there exist a neighborhood $V$ of $x$ in $S$
 such that 
 $$
 V\cap U_-\subset E_{n+1}.
 $$
 By~\eqref{buiu} $\partial U_+\subset U_+$, therefore
 $$
 x\in\partial U_+=U\cap\partial F_0\subset U_+\cap F_n
 \subset U_+\cap int_S F_{n+1}.
 $$
  Then there is a neighborhood
 $W$ of $x$ in $S$ such that $W\subset F_{n+1}$.
 Thus
 $$
 W\cap U_+\subset U_+\cap F_{n+1}.
 $$
 Since $x\in G_n\subset U$, we can assume that $V\cup W\subset U$.
 Hence 
 $$
 V\cap W=((V\cap W)\cap U_-)\cup ((V\cap W)\cap U_+)\subset E_{n+1}\cup (U_+\cap F_{n+1})=G_{n+1}.
 $$
 \qed
 \end{enumerate}

 \begin{Lemma}\label{L16}\quad
 \begin{enumerate}
 \item\label{l161}
 Every component of $U-G_n$ is contained in $U_+$ or in $U_-$.
 \item\label{l162}
 The components of $U-G_n$ contained in $U_-$ are the components of $U_--E_n$.
 \item\label{l163}
 The components of $U-G_n$ contained in $U_+$ are the components 
 of $S-F_n$ 
 \linebreak 
 contained in $U$.
 \end{enumerate}

 \end{Lemma}

\begin{proof}
\quad

Recall that $G_n=E_n\cup (U_+\cap F_n)\subset E_n\cup (F_n-F^\circ_0)$. 

Let $V$ be a connected component of $U-G_n$.
Observe that the open surfaces $U_+-\partial U_+$ and $U_--\partial U_-$
are disjoint with $\partial U_+=\partial U_-$.
Since $V$ is connected, if we had that
\linebreak
 $V\cap U_+\ne \emptyset$ and $V\cap U_-\ne \emptyset$,
then $V$ would intersect
 $$
 \partial U_+=\partial U_-=U_+\cap U_-=U\cap \partial F_0
 \subset U_+\cap F_n\subset G_n
 \qquad\text{ using~\eqref{buiu},}
 $$
 a contradiction with $V\subset U-G_n$. This proves item~\eqref{l161}.
 
 Before proceeding, note that 
 \begin{enumerate}[(i)]
 \item\label{ii1}
 $U_--E_n\subset U-G_n$. 
 \newline
 In fact, let $x\in U_--E_n$. If we had that $x\in G_n$, then 
 $x\in U_+\cap F_n$ and then 
 $$
 x\in U_+\cap U_- = U\cap \partial F_0=\partial U_-\subset E_n.
 $$
  A contradiction.
 \item\label{ii2}
 We also have that every component of $S-F_n$
 is contained in $U$ or is disjoint from $U$.
 This follows from the fact that $U$ is a component
 of $S-K$ which contains $S-F_n$.
 \end{enumerate}

 Let $V$ be a component of $U-G_n$ contained in $U_-$.
 Since $V\cap G_n=\emptyset$
 we have that
  $V\cap E_n=\emptyset$.
  Therefore $V$ is 
 contained in a component of $U_--E_n$.
 Conversely, suppose that $W$ is a component of $U_--E_n$.
 By remark~\eqref{ii1},
 $U_--E_n\subset U-G_n$.
 This  implies that $W$ is contained
 in a component of $U-G_n$.
 This proves item~\eqref{l162}.
 
  Observe that to prove item~\eqref{l163} it is enough to prove that 
 any component of $U-G_n$ contained in $U_+$ is contained in a 
 component of $S-F_n$ contained in $U$ and vice versa.

 Now let $V$ be a component of $U-G_n$ contained in $U_+$.
 Then 
 \begin{equation}\label{vufn}
 V\cap (U_+\cap F_n)=\emptyset.
 \end{equation}
 By the choice of $V$, $V\subset U_+$. Thus from~\eqref{vufn}, $V\cap F_n=\emptyset$.
 Therefore $V$ is contained in a component $W$
 of $S-F_n$.
 Since $W\cap U\ne \emptyset$, by remark~\eqref{ii2} we have that $W\subset U$.

 Conversely, let $W$ be a component  of $S-F_n$ contained in $U$.
 Since by~\eqref{uu}, $U_-\subset F_n$, we have that
  \begin{equation}\label{wu-}
  W\cap U_-=\emptyset
  \qquad\text{ and }\qquad
  W\subset U_+.
  \end{equation}
 We  claim that $W\cap G_n=\emptyset$. 
  If we had a point $x\in W\cap G_n$, then $x\in G_n$ 
  implies that $x\in U_+\cap F_n$
  or $x\in E_n$.
  Since $x\in W$, we have that $x\in U_+\cap F_n$ is not possible.
  If we had $x\in E_n$, then $x\in U_-$.
  Therefore, using~\eqref{wu-}, 
  $x\in W\cap U_-=\emptyset$, a contradiction.
  This proves the claim.
  Therefore $W\subset U- G_n$ and then $W$
  is contained in a component of $U-G_n$.
  And by~\eqref{wu-}, $W\subset U_+$.
  This proves item~\eqref{l163}.
 
\end{proof}

Lemma~\ref{L16} has two purposes.
First, it concludes the proof that $(G_n)$ is a exhaustion of $U$
by showing that the closure in $U$ of every component of $U-G_n$
is a bordered surface whose boundary consists of exactly one
component of $\partial G_n$.
This follows from the fact that $(F_n)$ 
and $(E_n)$ are exhaustions of $S$ and $U_-$, 
respectively.

The second is the relation between ideal boundary points of $U$ and ideal
boundary points of $S$ and $U_-$,
namely, 
$$b(U)=b(U_+)\cup b(U_-)
\quad\text{  with }\quad 
b(U_+)\subset b(S).
$$

\begin{Proposition}\label{P17}
\quad

Let $U$ be a residual domain of a compact set $K$
contained in a connected boundaryless surface $S$ and assume that $K$ 
has finitely many connected components.
\begin{enumerate}
\item\label{l171}
Then $U$ has only finitely many ideal boundary points that are relatively
compact in $S$. Each of these points is isolated in $b(U)$.
\item\label{l172}
The ideal boundary points of $U$ that are not relatively compact in $S$
are ideal boundary points of $S$ and their impressions relative to  $S$
are disjoint from $S$.
\item\label{l173}
The frontier of $U$ in $S$ is the union of the impressions in $S$
of the ideal boundary points of $U$ that are relatively compact in $S$.
\end{enumerate}

\end{Proposition}

\begin{proof}
Consider first the case when $U$ is relatively compact.
As in Lemma~\ref{L16}, there exists a compact bordered
surface $F_0$ which is a neighborhood of $K$ in $S$ and
all relatively compact residual domains of $K$ are contained in $F_0$.
It follows from Proposition~\ref{P15}
 that $U$ has finitely many ideal boundary points. 
 Of course, each one is isolated and relatively compact in $S$.
 From Proposition~\ref{P12}, the ideal boundary points of $U$
 are planar and the ideal completion of $U$ is a compact surface.
 
 We use a complete metric on $S$.
 Now consider the case $U$ is unbounded and let $b$ be an ideal
 boundary point of $U$.
 As before, we decompose $U$ into $U_+\cup U_-$.
 
 From Proposition~\ref{P9} and item~\eqref{l161} of Lemma~\ref{L16},
 $b$ can be represented by an ideal boundary component $(V_n)$,
 where the sets $V_n$ are components of $U-G_n$ and 
 either all of them are contained in $U_+$ or all of them are contained in $U_-$.
 Let $b_+(U)$ be the set of ideal boundary points of $U$ 
 for which $V_n\subset U_+$ for every $n$
 and let $b_-(U)$ be the set of ideal boundary points of $U$ for which 
 $V_n\subset U_-$ for every $n$.
 
 Recall that  $F_n$ is an exhaustion of $S$.
 If $b\in b_+(U)$, then from item~\eqref{l163} of Lemma~\ref{L16} we see that 
 $(V_n)$ is an ideal boundary component defining a point $b'\in b(S)$.
 Moreover $Z(b)=\cap_n cl_{B(S)}V_n=\{b'\}$.
 This proves item~\eqref{l172}.
 
 Suppose now that $b\in b_-(U)$.
 In this case $(V_n)$ also defines an ideal boundary point
 $b'$ of $U_-$, and this correspondence defines a bijection from
 $b_-(U)$  onto $b(U_-)$.
 By Claim~\ref{24c2} the set $U_-$ is a residual domain of $K$ in
 the compact manifold with boundary $F_0$.
 From Proposition~\ref{P15} we have that $b(U_-)\approx b_-(U)$ is finite.

 Clearly every $b\in b_-(U)$ is relatively compact in $S$.
 Since every $b\in b_+(U)$ is not relatively compact in $S$,
 we have that $b_-(U)$ is the set of relatively compact
 ideal boundary points of $U$.
 An ideal boundary point $b\in b_-(U)$ can not be accumulated by
 ideal boundary points of $b_+(U)$,
 and since $b_-(U)$ is finite, we have that 
 its points are isolated in $b(U)$.
 This proves item~\eqref{l171}.
 
 We now prove item~\eqref{l173}, 
 let $b=(V_n)$  be a relatively compact ideal boundary point of $U$.
 Then $b\in b_-(U)$ and  $Z(b)=\cap_n cl_S V_n$.
 If $x\in Z(b)$, then $x\in cl_SU$.
  If we had $x\in U$, 
  then $x\in U\cap cl_S V_n= cl_U V_n$ for every $n$,
  contradicting that $(V_n)$ is an ideal boundary
  component of $U$.
 Therefore $Z(b)\subset fr_S U$.
 
 It remains to show that  $fr_SU\subset \cup_{b\in b_-(U)}Z(b)$.
  
 Let $b_1,\ldots,b_m$ be the relatively compact
  ideal boundary ponts of $U$.
  Each $b_i$ can be represented by an ideal boundary component
  $(V_n^i)$ of $U_-$, where $V_n^i$ is a component of $U_--E_n$.
  Since $b(U_-)$ is a finite set with $m$ elements,
 there exists $n_0$ such that for $n\ge n_0$ the number of components
 of $U_--E_n$ is also $m$.
 For each $i=1,\ldots,m$, the collection $(V^i_n)_{n\ge n_0}$
 is an ideal  boundary component of $U$ that represents $b_i$, 
 where $V^1_n,\ldots,V^m_n$ are the components of $U_--E_n$.

 Now let $x\in fr_SU$ and let $(x_k)$ be a sequence in $U$ such that 
$\lim x_k=x$. Given $n_1\ge n_0$ there exists $k_0=k_0(n_1)$ such that 
$x_k\notin G_{n_1}$ for $k\ge k_0$. Since $fr_SU\subset K\subset F^\circ_0$,
there is $k_1=k_1(n_1)\ge k_0$ such that $x_k\in U\cap F_0$ for $k\ge k_1$.
Using~\eqref{uu} and \eqref{Gn2} we have that 
$ x_k\in U_--E_{n_1}$ for $k\ge k_1$. This is 
 \begin{equation}\label{k1n1}
 \forall n_1\ge n_0 \quad \exists k_1=k_1(n_1) \quad
 \forall k\ge k_1(n_1) \qquad 
 x_k\in U_--E_{n_1}.
 \end{equation}

Using $n_1=n_0$ in \eqref{k1n1} we have that there exists $i_0$ and a subsequence 
$(x_{k_j})$ of $(x_k)$ such that $x_{k_j}\in V^{i_0}_{n_0}$ for every $j$.
For $n\ge n_0$ we have that $V^{i_0}_n$ is the only component
 of $U_--E_n$ contained in $V^{i_0}_{n_0}$.
 Using~\eqref{k1n1} with $n_1=n\ge n_0$ we have that there exists
 $j_n$ such that $x_{k_j}\in V^{i_0}_n$ for all $j\ge j_n$.
 It follows that $x\in cl_S V^{i_0}_n$ for every $n\ge n_0$
 and then
 $$
 x\in\cap_{n\ge n_0} cl_S V^{i_0}_n=Z(b_{i_0}).
 $$

 \end{proof}

There is a natural inclusion $b(S)\subset b(S-K)$.
Namely, if $b=(V_n)\in b(S)$ then 
$$
\exists n_0 \quad \forall n\ge n_0\quad
 V_n\cap K=\emptyset, \quad \text{ and hence }\quad b'=(V_n)_{n\ge n_0}\in b(S-K).
 $$
 
 The set $S-K$ may have infinitely many components and $b(S-K)-b(S)$ may be infinite,
 but
 
\begin{Corollary}\label{C325}\quad

Let $S$ be a connected boundaryless surface.
Let $K$ be a compact subset of $S$ with finitely many
 connected components.
 \begin{enumerate}
 \item\label{c3251} $b(S-K)-b(S)$ is a discrete topological space.  \item\label{c3252} If $U$ is a residual domain in $S-K$ then $b(U)-b(S)$ is finite.
 \end{enumerate}
\end{Corollary}

\begin{proof}\quad

By item~\eqref{l172} in Proposition~\ref{P17},
 the ideal boundary points in $b(U)-b(S)$ are relatively compact.
By item~\ref{P17}.\eqref{l171}, $b(U)$ has only finitely many relatively compact ideal boundary points.
Therefore $b(U)-b(S)$ is finite. This proves item~\eqref{c3252}.

Item~\eqref{l172} in Proposition~\ref{P17} implies that the ideal boundary points
in $b(S-K)-b(S)$ are relatively compact. 
Then 
\begin{equation}\label{diju}
b(S-K)-b(S)=\sqcup_{\la\in\La} (b(U_\la)-b(S)),
\end{equation}
 where $\La$ is the collection of 
 residual domains in $S-K$. 
By item~\eqref{c3252} each $b(U_\la)$ is finite and discrete.
And $b(S-K)-b(S)$ has the disjoint union topology in~\eqref{diju}.
This proves item~\eqref{c3251}.

\end{proof}

Summarizing for future reference we have 

 \begin{Proposition}\label{P16}\quad
 
 Let $S$ be a connected surface without boundary.
 Let $K$ be a compact subset of $S$ with finitely many
 connected components and  let $U$ be
 a residual domain of $K$, i.e. a connected component of $S-K$.
 There is a decomposition 
 $$
 U = U_+\cup U_-,
 $$
 where $U_+$ and $U_-$ are connected surfaces with compact boundary
 $\partial U_+=\partial U_-=U_+\cap U_-$ which satisfy
 $$
 b(U)=b(U_+)\sqcup b(U_-)
 \quad \text{ with } \quad
 b(U_+)=b(S)\cap b(U).
 $$
 
 More explicitly,
 there are exhaustions  $(E_n)$, $(F_n)$, $(G_n)$ of $U_-$, $S$ and $U$     respectively
 such that 
 \begin{enumerate}
 \item\label{p16b1}
 $U_- = U\cap F_0$ \quad and \quad  $U_+=U-F^\circ_0$.
 \item\label{p16b2} $F_0^\circ$ contains $K$ and all the relatively compact residual domains in $S-K$.
 \newline
 In particular $K\subset cl_S U_-$.
 \item\label{p16b3}
 Every component of $U-G_n$ is contained in $U_+$ or in $U_-$.
 \item\label{p16b4}
 The components of $U-G_n$ contained in $U_-$ are the components of $U_--E_n$.
 \item\label{p16b5}
 The components of $U-G_n$ contained in $U_+$ are the components 
 of $S-F_n$ 
 \linebreak 
 contained in $U$.
 \end{enumerate}
 
 Moreover, 
 \begin{enumerate}[(a)]
 \item\label{p16ba} The interior of the surface with boundary $F_0$ contains $K$ and all the relatively 
 compact residual domains of $K$.
 \item\label{p16bb} The relatively compact ideal boundary points of $U$ are finitely many 
 and are isolated in $b(U)$. Moreover, they are exactly the ideal boundary points
 in $b(U_-)$ of the surface $U_-=U\cap F_0$.
 \item\label{p16bc} The ideal boundary points in the complement $U_+$ 
 of the interior of the surface $F_0$, 
 $U_+=U-F_0^\circ$ are in $b(S)$ and hence they are not relatively compact in $S$.
 \item\label{p16bd} Using Proposition~\ref{P12}, for any relatively compact ideal boundary point of $U$, 
 $b\in b(U_-)$ there is a neighborhood of $b$ in $B(U)$ given by $V_n^*=V_n\cup\{b\}$, where
 $V_n$ is a component of $U_--E_n$, such that $V_n^*$ is homeomorphic to a disk in $\re^2$.
 \item\label{p16be} The frontier of $U$ in $S$ is the union of the impressions in $S$
of the ideal boundary points of $U$ that are relatively compact in $S$.
\item\label{p16bf} By items~\eqref{p16bd} and~\eqref{p16be}, if $x\in fr_S U$
then $x\in Z(b)$, where $b\in b(U)$ is relatively compact 
and there is a neighborhood of $b$ in $B(U)$
which is a disk.
  \end{enumerate}
 \end{Proposition}

As a consequence of Proposition~\ref{P17} we have the following:

\begin{Corollary}\label{C18}
\quad

Let $U$ be an open connected subset of a connected surface without boundary $S$.
Then $fr_SU$ is a compact set with finitely many connected components
if and only if $U$ is a residual domain of a compact set $K$
that has finitely many connected components.
\end{Corollary}

\pagebreak
\begin{proof}\quad

If $fr_SU$ is a compact set with finitely many connected components,
then $U$ is a residual domain of $S-fr_SU$.
Conversely, 
if $U$ is a residual domain of a compact set $K$
that has finitely many connected components,
then by item~\eqref{l173} of Proposition~\ref{P17}, $fr_SU=\cup_bZ(b)$,
where the union is taken over the finite set
of relatively compact ideal boundary points $b$ of $U$.

\end{proof}

This result is false if $S$ is not a surface.
Consider the following subset of $\re^2$:
$$
X=([0,1]\times\{0,2\})\cup(C\times[0,2]),
$$
where $C$ is an infinite closed subset of $[0,1]$.
If $K=\{(x,y)\in X\;:\;y\le 1\}$,
then $K$ is compact connected,
$U=X-K$ is connected,
but $fr_XU=C\times \{1\}$.

\section{The accumulation lemma.}

Let $S$ be a connected surface without boundary and let $S_0\subset S$
be an open subset with $fr_S S_0$ compact.
Let $\mu$ be a Borel measure in $S$ which  is finite on compact sets
and positive on open non empty sets. We call $\mu$ the area measure.
Let $f:S_0\to S$ be an injective area preserving homeomorphism of $S_0$
onto an open subset $f(S_0)\subset S$. Let $K\subset S_0$ be a connected
compact subsets such that $f(K)=K$.

 There is a natural inclusion $b(S)\subset b(S-K)$ obtained as follows.
 If $(V_n)$ is an ideal boundary component for $q\in b(S)$, since $K$ is compact,
 there is $n_0$ such that   $V_n \cap K=\emptyset$ for all $n>n_0$.
 Then $(V_n)_{n>n_0}\in b(S-K)$.
  
  Let $p\in b(S-K)$,  from Proposition~\ref{P17}.\eqref{l172} we have that
   \renewcommand{\theenumi}{\theequation}
 \def\iitem{\refstepcounter{equation}\item}
  \begin{enumerate}
  \iitem  If $p\in b(S)$ then $Z(p)=p$. 
  \iitem\label{ZpK} Otherwise $p\in b(S-K)-b(S)$, $Z(p)\subset K$ and $p$ is relatively compact.
  \end{enumerate}

 We claim that the map $f$ naturally defines an injective  map
$$
f_*:b(S-K)-b(S)\to b(S-K)-b(S).
$$
Indeed, if $p=(V_n)\in b(S-K)-b(S)$ 
then $Z(p)\subset K$, therefore for all $n$, $cl_S V_n\cap K\ne\emptyset$.
Since $fr_S S_0$ is compact 
and $fr_S S_0\cap K=\emptyset$,
$\exists n_1, \; \forall n>n_1, \;V_n\cap fr_S S_0=\emptyset$.
Since $V_n$ is connected and $K\subset S_0$, we obtain that $\forall n>n_1$,
 $V_n\subset S_0$.
Therefore $fV_n$ is defined for $n>n_1$ 
and $(fV_n)_{n>n_1}$ is an ideal boundary component for a 
point $f_*(p)$ which is  in $b(S-K)-b(S)$, because $Z(f_*(p))=f(Z(p))\subset K$.
The injectivity of $f$ implies that $f_*$ is injective.

\begin{Proposition}\label{P51}
For any $p\in b(S-K)-b(S)$ there is $n$ such that $f_*^n(p)=p$.
\end{Proposition}

\begin{proof}
\quad

By Proposition~\ref{P13}, there exists a compact
bordered surface $N$ which is a neighborhood 
of $K$ in $S_0$, such that all relatively compact residual
domains in $S_0-K$ are contained in $N$.
The residual domains in $S-K$ which are not contained
in $N$, contain a component of $\partial N$.
Then there are only finitely many of them.
Let $\ve$ be the minimum of the area of the 
residual domains in $S-K$ which are not contained in $N$.
Then $\ve>0$.

Since $\ve>0$ and the area of $N$ is finite, there are only finitely 
many residual domains in $S-K$ with area $\ge\ve$ which are contained in $N$.
The number of residual domain not contained in $N$ is finite. 
Therefore the number of residual domains in $S-K$ with area $\ge \e$
is finite.

If $U$ is a residual domain in $S-K$ with area $<\e$. Then $U$ is contained in 
$N\subset S_0$ and hence it is relatively compact in $S_0$ and $fr_SU\subset K$.
The map $f$ is defined in $U$ and $f(U)$ is a connected component of $f(S_0)-K$
which is relatively compact in $f(S_0)$. Then the frontier $fr_Sf(U)\subset K$.
This implies that $f(U)$ is a connected component of $S-K$. Since $f$ preserves
area, we have that $f(U)$ is a residual domain in $S-K$ with area $<\e$, and then
$f(U)\subset N\subset S_0$. Inductively, all the iterates $f^n$ are defined in $U$ and
$f^n(U)$ is a residual domain in $S-K$ with $f^n(U)\subset N$. Since the area of $N$
is finite, we have that there is $n>0$ such that $f^n(U)=U$. This implies that the map
$f^n_*$ sends $b(U)$ to itself.

If $U$ is a residual domain in $S-K$  with area $<\ve$ then it is relatively
compact in $S$ and hence by  Proposition~\ref{P17}.\eqref{l171},
$b(U)$ is finite. Since $b(U)$ is a periodic set for $f_*$ and $f_*$ is injective,
we obtain that all points in $b(U)$ are periodic for $f_*$.

 There are only finitely many residual domains $V$ in $S-K$ of area $\ge\ve$ 
 and by 
 \linebreak
 Corollary~\ref{C325}.\eqref{c3252} for each one
 $b(V)-b(S)$ is finite. 
 Therefore the set  $Q\subset b(S-K)-b(S)$ of points which are not periodic for $f_*$
 is at most finite. The injectivity of $f_*$ implies that $f_*(Q)\subset Q$.
 Since $f_*$ is injective, $Q$ is finite and $f_*(Q)\subset Q$, 
 we have that $f_*$ is a bijection, i.e. a permutation of $Q$. Then all points in $Q$ are periodic, i.e. $Q=\emptyset$,
  and all points in $b(S-K)-b(S)$ are periodic for $f_*$.
 
 \end{proof}

 We will need the following result.
 In its proof we refer to the canonical exhaustion of $U$ from subsection~\S~\ref{canonical}
 which is also described in Proposition~\ref{P16}.

\begin{Proposition}\label{P21}\quad

Let $U$ be a residual domain of $K$ and 
$\a:]0,1]\to U$ a path such that
$$\lim\nolimits_{t\to 0}\a(t)=p\in K\quad \text{ in }S.
$$
Then there exists a relatively compact 
ideal boundary point $b$ of $U$
such that 
$$
\lim_{t\to 0}\a(t)=b\quad\text{ in }B(U).
$$

\end{Proposition}

\begin{proof}\quad

Observe that $U_-$ contains a relative neighborhood of $K\subset fr_S U_-$ inside $U_-$.
Then there exists $\tau>0$ such that $\a(t)\in U_-$ for all $t<\tau$.
As in the proof of item~\eqref{l173}  of Proposition~\ref{P17},
if $b_1,\ldots,b_m$ are the relatively compact
ideal boundary points of $U$,
then there exists $n_0$ such that  
for $n\ge n_0$ the number of components
of $U_--E_n$ is $m$, and for each
$i=1,\ldots,m$, the collection
$(V^i_n)_{n\ge n_0}$ is an ideal
boundary component of $U$
that represents $b_i$,
where $V^1_n,\ldots,V^m_n$
are the components of $U_--E_n$.
There exist $0<t_0<\tau$ and $i_0$ such that 
the curve
$\a(]0,t_0[)\subset V^{i_0}_{n_0}$.
For every $n>n_0$ we have that 
$V^{i_0}_n$ is the only
component of $U_--E_n$ contained in $V^{i_0}_{n_0}$,
and therefore there exists a sequence
$t_n\downarrow 0$
such that $\a(]0,t_n[)\subset V^{i_0}_n$.
For $n>n_0$, we have that 
$$
cl_{B(U)}\a(]0,t_{n+1}[)\subset cl_{B(U)}V^{i_0}_{n+1}\subset (V^{i_0}_n)^*:=V^{i_0}_n\cup\{b_{i_0}\},
$$
implying that
 $$
 \cap_{n> n_0}cl_{B(U)}\a(]0,t_{n+1}[)\subset \cap_{n>n_0}(V^{i_0}_n)^*=\{b_{i_0}\}.
 $$
 From this we conclude that $\lim_{t\to 0}\a(t)=b_{i_0}$ in $B(U)$.

\end{proof}

 Let $S$ be a connected surface with 
  compact boundary  and  $S_0\subset S$ be an open subset with $fr_S S_0$ compact.
 Let $f:S_0\to S$ be an area preserving 
 homeomorphism onto its image.
Let $p\in\text{Fix}(f^n)$ be a periodic point of $f$, possibly in the boundary $p\in \partial S_0$.
We say that $p$ is a saddle point if there is an open neighborhood $V$
of $p$ in $S_0$ and a continuous system of coordinates in $V$ such that 
in these coordinates we have that $f^n(x,y)=(\la x, \la^{-1} y)$ for some
$\la>1$. In this case the {\it branches} of $p$ are the connected
components of $W^s(p,f^n)- \{p\}$ and of $W^u(p,f^n)-\{p\}$,
where
\begin{align*}
W^s(p,f^n)&=\textstyle\big\{\,y\in S:\lim_{k\to +\infty} f^{nk}(y)=p\,\big\},
\\
W^u(p,f^n)&=\textstyle\big\{\,y\in S:\lim_{k\to-\infty}f^{nk}(y)=p\,\big\}
\end{align*}
are the {\it stable} and {\it unstable manifolds of $p$}.
The {\it branches of $f$} are the branches of its saddle periodic points.

In order to use the previous results for surfaces without boundary we define 
the {\it double} of a surface with boundary. 
Let  $\sim$ be the equivalence relation defined by the partition of 
$S\times\{0,1\}$ into one point sets $\{(p,i)\}$ if $p\notin\partial S$, $i\in\{0,1\}$
and two point sets $\{(p,0),(p,1)\}$ if $p\in\partial S$.
Let $\ov S = S\times\{ 0, 1\}/\sim$, provided with the quotient topology.
Then $\ov S$ is a surface without boundary and $\ov S_0 =S_0\times\{0,1\}/\sim$
is an open subset of $\ov S$.

Let $\nu$ be the measure with $\nu(\{0\})=\nu(\{1\})=1$ on $\{0,1\}$.
Extend the measure $\mu$ in $S$ to the measure $\ov \mu$ given by the 
push forward of the measure $\mu\times\nu$ on $S\times\{0,1\}$ by the quotient map.
Extend the map $f$ to $\ov S_0$ by $\ov f[(x,i)]=[(f(x),i)]$, where $[(x,i)]$ is the class of $(x,i)$.
Then $\ov f$ is an area preserving homeomorphism of $\ov S_0$ onto an 
open subset of $\ov S$. If $p$ is a saddle point for $f$ in $S$ and $i\in\{0,1\}$,
then the point $[(p,i)]\in\ov S$ is a saddle point for $\ov f$.

\begin{Theorem}[The accumulation lemma]\label{P22}\quad

 Let $S$ be a  connected surface with compact boundary provided with 
a Borel measure $\mu$ such that open non-empty subsets have positive
measure and compact subsets have finite measure.
Let $S_0\subset S$ be an open subset with $fr_S S_0$ compact.

Let $f, f^{-1}:S_0\to S$ be an area preserving homeomorphism of $S_0$
onto  open subsets $f(S_0)$, $f^{-1}(S_0)$ of $S$.
Let $K\subset S_0$ be a compact connected  invariant subset of $S_0$.

If $L\subset S_0$ is a branch of $f$ and $L\cap K\ne \emptyset$,
then $L\subset K$.
\end{Theorem}

 \begin{proof}\quad

By applying the theorem to the double of the surface $\ov S$ and $\ov S_0$, $\ov f$,
we can assume that $\partial S=\emptyset$.
Observe that $\ov S_0$ is an open subset of $\ov S$ and
$fr_{\ov S}\ov S_0$ is also compact.
We use $\ov L=\pi_0(L)$ and $\ov K=\pi_0(K)$, where 
$\pi_0:S\to\ov S$ is $\pi_0(x)=[(x,0)]$.
We may assume that $L$ is an invariant stable branch.
Observe that it is enough to prove the theorem for 
an iterate $f^q$ of $f$.

Suppose by contradiction that $L$ contains an arc $\be$
whose endpoints $p_0$ and $p_1$ belong to $K$
but its interior $\be^\circ$ satisfies $\be^\circ\cap K=\emptyset$.
We are going to think of $\be$ also as a path
$\be:[0,1]\to U$, with $\be(0)=p_0$ and $\be(1)=p_1$.
Since $L$ is invariant and $\be\subset L\subset S_0$, 
we have that $f^n\be$ is defined and 
$f^n\be\subset S_0$ for all $n$.

Let $U$ be the connected component of $S-K$ that contains $\be^\circ$.
By Proposition~\ref{P17}.\eqref{l172} the relatively compact ideal boundary 
points for $U$ are in $b(S-K)-b(S)$.
By Proposition \ref{P17}.\eqref{l171} they are finitely many and by Proposition~\ref{P51}
they are periodic points for $f_*$.
Therefore by taking an iterate $f^q$ if necessary,
we can assume that 
 $f_*(b)=b$ for any relatively compact ideal boundary point of $U$.

By Proposition~\ref{P21}, there exist two relatively
compact ideal boundary points of $U$, $b_0$ and $b_1$,
such that $\lim_{t\to i}\be(t)=b_i$ in $B(U)$, $i=0,1$.
If $\be_*=cl_{B(U)}\be^\circ$, then 
$\be_*=\be^\circ\cup\{b_0,b_1\}$.
By Proposition~\ref{P17}.\eqref{l171} we have that 
$b_0$ and $b_1$ are isolated points of $b(U)$.

Referring to the canonical exhaustion $(G_n)$ of $U$,
described in Proposition~\ref{P16}, \eqref{p16bb}, \eqref{p16b4},
the point
$b_0\in b(U)$ is represented by an ideal boundary component $(V_k)$ of $U$,
where $V_k$ is a component of $U_--E_k$
and $cl_SV_k$ is a bordered surface whose boundary contains 
only one component, which we denote by
$\partial(cl_U V_k)$.
By~\ref{P16}.\eqref{p16bb} or \ref{P17}.\eqref{l171},
for every $k$ sufficiently large,
 $b_0$ is the only ideal boundary point of $U$ contained in $cl_{B(U)}V_k$.
 So we assume that this happens for every $k$.
  Let $V^*_k:=V_k\cup\{b_0\}$. By~\ref{P16}.\eqref{p16bd}, the collection $(V^*_k)$
  is a fundamental system
 of neighborhoods of $b_0$ in $B(U)$,
$D_k:=cl_{B(U)}V^*_k$ 
 is a closed disk and
 $\partial D_k=\partial (cl_U V_k)$.

Since $fr_S S_0$ 
is compact and $K\subset  S_0$, we have that
$U\cap fr_S S_0$ is compact. Then
\begin{equation}\label{vkfrs0}
\exists k_0\quad \forall k> k_0\quad V_k\cap fr_S S_0=\emptyset.
\end{equation}
Using~\eqref{Zb} and~\eqref{ZpK} we have that 
\begin{equation}\label{ivk}
\emptyset\ne\textstyle\bigcap_k cl_{B(S)} V_k=Z(b_0)\subset K.
\end{equation}
 From~\eqref{ivk}, $\forall k$ $cl_S V_k\cap K\ne\emptyset$.
 Since $V_k$ is connected and $K\subset S_0$, from~\eqref{vkfrs0} we get that 
 \linebreak
 $\forall k> k_0$ $V_k\subset S_0$.
Fix $k_1$ such that $V_{k_1}\subset S_0$ and $D_0:=cl_{B(S)}V_{k_1}^*$ is a closed disk and
\begin{equation}\label{D0S}
\forall k\ge k_1\qquad V_k^*\subset int_{B(S)} D_0\subset V_{k_1}\cup\{b_0\}.
\end{equation}
Extend the measure $\mu$ to $B(U)$ by $\mu(b(U))=0$.
Then $f_*$ is defined in $D_0$, is continuous,  preserves area and $f_*(b_0)=b_0$.

By its definition in~\eqref{uu} we have that $U_-$ is relatively compact in $S$,
and hence it has finite measure. 
By Proposition \ref{P16}.\eqref{p16bb} $V_k\subset U_-$.
Since $\mu(U^*_-)=\mu(U_-)<\infty$, $V_k\subset U_-$ and $\cap_k V_k^*=\{b_0\}$, 
we have that 
\begin{equation}\label{Vn0}
\lim_{k\to\infty}\mu(V_k^*)=0.
\end{equation}

Let $d$ be a metric compatible with the topology of $S$.
Since $\be^\circ\subset L$,
if $\diam(f^n\be)$ is the diameter of $f^n\be$,
then $\lim_{n\to\infty}\diam(f^n\be)=0$.
Therefore
\begin{equation}\label{nk0}
\forall k\quad\exists n_k>k\quad \forall n\ge n_k\quad
\diam(f^n\be)<d(\partial(cl_U V_k),K),
\end{equation}
where $d(\partial(cl_U V_k),K)=\inf\{ d(x,y):x\in \partial(cl_UV_k),\,y\in K\}$.

Since $K$ is invariant, we have that $f^n\be(0)\in K$.
Therefore $f^n\be$ can not intersect
$\partial (cl_UV_k)$ if $n\ge n_k$.
Since $\lim_{t\to 0}f^n\be(t)=f^n_*(b_0)=b_0$ in $B(U)$,
we have that for $n\ge n_k$, $f^n\be(t)\in V_k$ for some $t>0$.
This implies by connectedness that $f^n\be^\circ=f^n\be(]0,1[)\subset V_k$
and hence $b_0=b_1$.

For  $k>k_1$, $n\ge n_k$ we have that $f^n_*\be_*$ 
is a simple closed curve contained in $V^*_k$
which separates $D_0$ and also $B(U)$ in two connected components.
Let $A_n$ be the component of $B(U)-f_*^n\be_*$ contained in $V^*_k$.
Let $B_n:= D_0-cl_{B(U)}A_n$ be the other component of $D_0-f_*^n\be_*$.
We have that $A_n$ is homeomorphic to an open ball.

Since
$\lim_n\diam f_*^n(\be)=0$ we have that $\lim_n \diam A_n=0$
and then 
\begin{equation}\label{fanh}
\exists k_2>k_1\quad\forall n\ge n_{k_2}\quad f_*:A_n\to f_*(A_n)\subset D_0 \text{ is a homeomorphism.}
\end{equation}

Since $L$ is a stable branch, the set $cl_S(\cup_{n\ge 0} f^n \be)$ is compact, 
and hence it has finite $\mu$-measure.
Since the points in  $fr_{U} A_n\subset f^n\be\subset L$ are wandering
we have that 
$$
\mu(fr_U A_n)\le \mu(f^n\be)=\mu(\be)=0.
$$
Observe that $fr_{B(U)} A_n=fr_U A_n \cup \{b_0\}$,
and hence  $\mu(fr_{B(U)} A_n)=0$.
Therefore 
\begin{equation}\label{dabn}
\mu(D_0-A_n)=\mu(D_0-cl_{B(U)} A_n)=\mu(B_n).
\end{equation}

Since $V_{k_1}\subset D_0$, by \eqref{Vn0},
\begin{equation}\label{vkdo}
\exists k_3>k_2\qquad \forall k\ge k_3\qquad \mu(V_k)<\mu(D_0-V_k).
\end{equation}

If $k\ge k_3$ and $n\ge n_{k}$,
   then $A_n\subset V_k^*$. Therefore by~\eqref{Vn0}
we can extract a subsequence $A_n$ such that $\mu(A_n)$ is decreasing
and
\begin{equation}\label{Anto0}
\lim_n\mu(A_n)=0.
\end{equation}
Using~\eqref{vkdo} and \eqref{dabn}, we have that 
$$
\forall n\ge n_{k_3}\qquad \mu(A_n)<\mu(D_0-A_n)=\mu(B_n)\le \mu(B_{n+1}).
$$

Using~\eqref{fanh} we have that for $n\ge n_{k_3}$,

\begin{enumerate}
\iitem\label{ee24} $f_*(A_n)$ is connected.
\iitem\label{ee25} $f_*(A_n)\subset B(U)-f^{n+1}\be_*$.
\iitem\label{ee19} $fr_{B(U)} f_*(A_n)\subset f_*(fr_{B(U)} A_n) =f_*^{n+1}\be_*$.
\iitem\label{e15} $\mu(f_*(A_n))=\mu(A_n)<\mu(B_n)\le \mu(B_{n+1})$.
\end{enumerate}
\begin{claim}\label{cc7}
If  $f_*(A_n)\not\subset A_{n+1}$  then 
$f_*(A_n)\cap D_0$ would contain $D_0-cl_{B(U)} A_{n+1}=B_{n+1}$.
\end{claim}
But the thesis in Claim~\ref{cc7} contradicts~\eqref{e15}. Therefore 
\begin{equation}\label{faan}
 f_*(A_n)\subset A_{n+1}.
\end{equation}

\noindent{\it Proof of Claim~\ref{cc7}:}\quad

Since $fr_{B(U)} A_{n+1}=f_*^{n+1}\be_*$, from~\eqref{ee24} and~\eqref{ee19} we have that 
either $f_*(A_n)\subset A_{n+1}$ or $f_*(A_n)\subset B(U)-A_{n+1}$.
We are assuming that $f_*(A_n)\not\subset A_{n+1}$ then $f_*(A_n)\subset B(U)-A_{n+1}$.
Moreover by~\eqref{ee25}, $f_*(A_n)\subset B(U)-cl_{B(U)} A_{n+1}$.
Since $b_0\in D_0\cap cl_{B(U)}f_*(A_n)$, there is
$$
x_0\in D_0\cap f_*(A_n)\subset D_0\cap f_*(A_n)-cl_{B(U)}A_{n+1}.
 $$
Given $x_1\in D_0-cl_{B(U)} A_{n+1}$ there is
 a path 
$\ga:[0,1]\to D_0 -cl_{B(U)} A_{n+1}$ with $\ga(i)=x_i$, $i=0,1$.
Let $t=\sup\{s\in[0,1]:\ga(s)\in f_*(A_n)\}$.
Since $x_0\in f_*(A_n)$, $t>0$.
If $t<1$ then by~\eqref{ee19},
 $$
 \ga(t)\in  fr_{B(U)}f_*(A_n)\subset f_*^{n+1}\be_*=fr_{B(U)}A_{n+1},
 $$
which contradicts the choice of $\ga$.
Therefore $x_1\in D_0\cap f_*(A_n)$.

\ted

From~\eqref{faan} and~\eqref{Anto0} we have that 
$$
\forall n>n_{k_3}\quad \mu(A_{n+1})\ge \mu(f_*(A_n))=\mu(A_{n})>0
\qquad\text{and}\qquad
 \lim_n \mu(A_n)=0.
$$
A contradiction.
This finishes the proof of Theorem~\ref{P22}.

\end{proof}

\nocite{Mat12}


\begin{thebibliography}{1}

\bibitem{AhSa}
Lars~V. Ahlfors and Leo Sario, \emph{Riemann surfaces}, Princeton Mathematical
  Series, No. 26, Princeton University Press, Princeton, N.J., 1960.

\bibitem{CM2}
Gonzalo Contreras and Marco Mazzucchelli, \emph{Existence of {B}irkhoff
  sections for {K}upka-{S}male {R}eeb flows of closed contact 3-manifolds},
  Geom. Funct. Anal. \textbf{32} (2022), no.~5, 951--979.

\bibitem{CO3}
Gonzalo Contreras and Fernando Oliveira, \emph{Homoclinics for geodesic flows
  on surfaces}, Preprint, 2022.

\bibitem{FLC}
John Franks and Patrice Le~Calvez, \emph{Regions of instability for non-twist
  maps}, Ergodic Theory Dynam. Systems \textbf{23} (2003), no.~1, 111--141.

\bibitem{kerekjarto}
B\'ela Ker\'ekj\'art\'o, \emph{Vorlesungen \"uber {T}opologie {I}}, Die
  Grundlehren der Mathematischen Wissenschaften, vol.~8, Springer, Berlin,
  Heidelberg, 1923.

\bibitem{Mat9}
{John N.} Mather, \emph{Invariant subsets for area preserving homeomorphisms of
  surfaces}, Mathematical analysis and applications, Part B, Academic Press,
  New York-London, 1981, pp.~531--562.

\bibitem{Mat12}
\bysame, \emph{Topological proofs of some purely topological consequences of
  {C}arath\'eodory's theory of prime ends}, Selected studies:
  physics-astrophysics, mathematics, history of science, North-Holland,
  Amsterdam-New York, 1982, pp.~225--255.

\bibitem{OC1}
Fernando Oliveira and Gonzalo Contreras, \emph{No elliptic points from fixed
  prime ends}, Preprint, 2022.

\bibitem{Rich}
Ian Richards, \emph{On the classification of noncompact surfaces}, Trans. Amer.
  Math. Soc. \textbf{106} (1963), 259--269.

\end{thebibliography}

\def\cprime{$'$} \def\cprime{$'$} \def\cprime{$'$} \def\cprime{$'$}
\providecommand{\bysame}{\leavevmode\hbox to3em{\hrulefill}\thinspace}
\providecommand{\MR}{\relax\ifhmode\unskip\space\fi MR }
\providecommand{\MRhref}[2]{%
  \href{http://www.ams.org/mathscinet-getitem?mr=#1}{#2}
}
\providecommand{\href}[2]{#2}

\end{document}